\documentclass[10pt,reqno]{amsart}
\usepackage{amsmath,amsthm,amsfonts,amssymb,amscd,latexsym, mathabx}
\usepackage{graphicx,tikz}
\usepackage{pgf}
\usepackage{enumerate}
\usepackage[utf8x]{inputenc}
 \usepackage{geometry}
\newtheorem{thm}{Theorem}[section]

\newtheorem{cor}{Corollary}[section]
\newtheorem{lemma}{Lemma}[section]

\theoremstyle{definition}
\newtheorem{rk}{Remark}[section]
\newtheorem{defi}{Definition}[section]
\newtheorem{exa}{Example}[section]
\def\N{\mathbb{N}}

\newcommand{\Ai}{\mathcal{A}}
\newcommand{\Bi}{\mathcal{B}}

\newcommand{\Fi}{\mathcal{F}}

\newcommand\blfootnote[1]{%
  \begingroup
  \renewcommand\thefootnote{}\footnote{#1}%
  \addtocounter{footnote}{-1}%
  \endgroup
}
\begin{document}
\title[On the image set and reversibility of shift morphisms]{On the image set 
and reversibility of shift morphisms over discrete alphabets}
\author[J. Campos]{Jorge Campos}
\address{Universidad Nacional Experimental Polit\'ecnica Antonio José de Sucre.
Departamento de Estudios Básicos. Sección de Matem\'atica. 
Barquisimeto, Venezuela.}
\email{jorgkmpos@gmail.com}
\author[N. Romero]{Neptal\'i Romero}
\address{Universidad Centroccidental
Lisandro Alvarado. Departamento
de Matem\'atica. Decanato de Ciencias y Tecnolog\'{\i}a.
Apartado Postal 400.
Barquisimeto, Venezuela.}
\email{nromero@ucla.edu.ve}
\author[R. Vivas]{Ram\'on Vivas}
\address{Universidad Nacional Experimental Polit\'ecnica Antonio José de Sucre.
Departamento de Estudios Básicos. Sección de Matem\'atica. 
Barquisimeto, Venezuela.}
\email{ramon.alberto.vivas@gmail.com}

\begin{abstract}
In this paper we provide sufficient conditions in 
order to show that the set image of a continuous and shift-commuting 
map defined on a shift space over an arbitrary discrete alphabet is also a 
shift space; additionally, if such a 
map is injective, then its inverse is also continuous and shift-commuting.
\end{abstract}

\maketitle

\section{Introduction}
\blfootnote{This work was partially supported by  
the Consejo de Desarrollo Científico, Humanístico y Tecnológico (CDCHT) of the 
Universidad Centroccidental Lisandro Alvarado under grant 1186-RCT-2019.}
%%%%%%%%%%%%%%%%%%%%%%%%%%%%%%%%%%%%%%%%%%%%%%%%%%%%%%%%%%%%
Shift spaces and their morphisms constitute a powerful tool for modeling several phenomena in dynamical systems, this practice is known as symbolic dynamics; they are also an important part of the platform for diverse disciplines like automata, coding, information and system theory. 
Depending on the context there are two typical ways to establish them: either on bi-infinite sequences of symbols indexed by the set of integer numbers $\mathbb{Z}$, or by using one-sided infinite sequences of symbols indexed by the set of natural numbers $\mathbb{N}$; in this paper we shall deal exclusively with the two-sided setting. Some classic textbooks on the subjects: \cite{kitchens} and \cite{lind}.

\smallskip
 
The standard construction of shift spaces begins with a finite set of symbols: an {\em alphabet}; however, in certain environments it is necessary to consider alphabets with infinite symbols, such is the situation when the thermodynamic formalism is developed for the so-called {\em countable state Markov shifts}, see for \cite{kitchens} and \cite{sarig}. Shift spaces over alphabets with infinite symbols were also considered by Gromov in his seminal work on endomorphisms of symbolic algebraic varieties and topological invariants of dynamical systems, see \cite{gromov} and \cite{gromov2}. These outstanding facts point to a deserved attention to the study of shift spaces
over infinite alphabets. 

\smallskip

In what follows an alphabet is any nonempty set $\mathcal{A}$ equipped with the discrete topology. 
Given an alphabet $\mathcal{A}$, the product space $\mathcal{A}^{\mathbb{Z}}$ is considered; this is the set of all bi-infinite sequences over $\mathcal{A}$ endowed with the Cantor metric, which is defined for all $x=(x_n)_{n\in\mathbb{\mathbb{Z}}}$ and $y=(y_n)_{n\in\mathbb{\mathbb{Z}}}$ in $\mathcal{A}^{\mathbb{\mathbb{Z}}}$ by
\begin{equation}\label{met}
d(x,y)=\begin{cases}0,\,\text{ if $x_n=y_n$ for all $n\in\mathbb{Z}$}
\\2^{-k},\,\text{ if $x\neq y$ and
$k=\min\{|n|: x_n\neq y_n\}$};\end{cases}
\end{equation}
this topological space is called the (two-sided) {\em full shift over $\mathcal{A}$}.
After that, it is contemplated the {\em shift operator} $\sigma:\mathcal{A}^{\mathbb{\mathbb{Z}}}\to\mathcal{A}^{\mathbb{\mathbb{Z}}}$ given by $\sigma(x)=y$, where $y_n=x_{n+1}$ for all $n\in\mathbb{\mathbb{Z}}$; it is a homeomorphism. Thus, a {\em shift space over $\mathcal{A}$} (or a {\em subshift over $\mathcal{A}$}) is any nonempty closed subset $X$ of $\mathcal{A}^{\mathbb{\mathbb{Z}}}$ which is shift-invariant, that is $\sigma(X)=X$; a {\em shift morphism} (or simply a {\em morphism}) on the shift space $X\subseteq \mathcal{A}^{\mathbb{Z}}$ is any continuous and {\em shift-commuting} mapping $\Phi$ from $X$ to some full shift $\mathcal{U}^{\mathbb{Z}}$; here the term shift-commuting means $\Phi\circ \sigma=\sigma\circ \Phi$, where $\sigma$ indistinctly denotes the shift map in both $\mathcal{A}^{\mathbb{Z}}$ and $\mathcal{U}^{\mathbb{Z}}$.

\smallskip

Shift spaces can be equivalently introduced by using a set of special words constructed with the symbols in the alphabet; more precisely, a {\em word} or a {\em block} over $\mathcal{A}$ is a finite sequence of symbols in $\mathcal{A}$, the number of such a symbols is the {\em length} of the block. Let $\mathcal{A}^*$ denote the set of blocks over $\mathcal{A}$; it is said that a block $w\in\mathcal{A}^*$ {\em appears} in $x\in\mathcal{A}^{\mathbb{Z}}$ if there exist an integer closed interval $[i,j]\subset\mathbb{Z}$ ($i\leq j$) such that the restriction $x\vert_{[i,j]}$ of $x$ to $[i,j]$ is just $w$; that is $x\vert_{[i,j]}=x_i\cdots x_j=w$. It is well known that a nonempty set $X\subseteq \mathcal{A}^{\mathbb{Z}}$ is a shift space over $\mathcal{A}$ if and only if there exists $\Fi\subset \mathcal{A}^*$ (not necessarily unique) in such a way that $x\in X$ if and only if no block belonging to $\Fi$ appears in $x$; such a set is called {\em a forbidden set for $X$}, obviously $\mathcal{F}=\emptyset$ if $X$ is a full shift. There is another set of distinguished words for each shift space. A word $w \in\mathcal{A}^*$ is called {\em allowed} for the shift space $X\subseteq \mathcal{A}^{\mathbb{Z}}$ if it appears in some point of $X$. The set $\mathcal{L}(X)$ of the allowed words for the shift space $X$ is called {\em language} of $X$, clearly 
$\mathcal{L}(X)=\{w\in\mathcal{A}^*: \text{ $w$ appears in $x$ for some $x$ in $X$}\}$.
The language $\mathcal{L}(X)$ is characterized by the {\em factorial} and {\em extendable} properties:
\begin{enumerate}[(a)]
\item
If $w$ is a block in $\mathcal{L}(X)$ and $u$ is a subblock of $w$, then $u\in\mathcal{L}(X)$.
\item 
If $w\in\mathcal{L}(X)$, then there are nonempty blocks $u,v\in\mathcal{L}(X)$ such that the 
concatenation block $uwv$ is also in $\mathcal{L}(X)$. 
\end{enumerate}
Properties (a) and (b) characterize the languages of shift spaces; that is, given a nonempty set $\mathcal{L}\subset\mathcal{A}^*$ satisfying (a) and (b), there is a unique shift space $X\subseteq\mathcal{A}^{\mathbb{Z}}$ such that $\mathcal{L}=\mathcal{L}(X)$. For every integer $n\geq 1$, $\mathcal{L}_n(X)$ denotes the subset of $\mathcal{L}(X)$ whose blocks have length $n$. It is not hard to show that $X$ is a compact space if and only if $\mathcal{L}_1(X)$ is a finite set; i.e. the alphabet supporting $X$ is a finite set. Notice that the complement of $\mathcal{L}(X)$ is a forbidden set for $X$; indeed, it is the biggest one. Obviously when $X$ is the full shift, $\mathcal{L}_n(X)=\mathcal{A}^n$.

\smallskip

There is a simple way to obtain morphisms on a shift space $X\subseteq \mathcal{A}^{\mathbb{Z}}$. If $\mathcal{U}$ is an alphabet, $m$ and $n$ are integers with $n-m\geq 0$ and   $\varphi:\mathcal{L}_{m+n+1}(X)\to\mathcal{U}$ is an arbitrary function, then the map $\Phi:X\to\mathcal{U}^{\mathbb{Z}}$ defined by
\begin{equation}\label{e1}
\Phi(x)_i=\varphi\left(x\vert_{[i-m,i+n]}\right)\,\text{for every 
$x$ in $X$ and each $i\in\mathbb{Z}$}
\end{equation}
is a morphism on $X$. Every map so defined is called {\em sliding block code}, the function $\varphi$ is known as a {\em local rule} inducing $\Phi$, and the integers $m$ and $n$ are called respectively {\em memory} and {\em anticipation} of $\Phi$. The local rule and the integers $m$ and $n$ provide a sort of window of length $n+m+1$, it slides through each element of $X$ to determine the value of each component of its image under the corresponding sliding block code. Sometimes it is useful to consider wider windows: take integers $M$ and $N$ with $M\geq m$ and $N\geq n$, the language factorial property implies that $\widehat{\varphi}:\mathcal{L}_{M+N+1}(X)\to \mathcal{U}$ with $\widehat{\varphi}(a_{-M}\cdots a_N) 
=\varphi(a_{-m}\cdots a_n)$ is well defined for every $a_{-M}\cdots a_N\in \mathcal{L}_{M+N+1}(X)$ and it induces the same sliding block code than $\varphi$. 

\smallskip

When $X$ is the full shift $\mathcal{A}^\mathbb{Z}$ and $\mathcal{U}=\mathcal{A}$, the concept of sliding block code is just the concrete definition of {\em cellular automaton} over the alphabet $\mathcal{A}$. Gustav A. Hedlund in his influential article \cite{hedlund} proved that if $\mathcal{A}$ is finite, then the set of cellular automata in $\mathcal{A}^\mathbb{Z}$ matches with the set of continuous and shift-commuting self-mappings of $\mathcal{A}^\mathbb{Z}$; so the cellular automata over finite alphabets were characterized. Since Hedlund credited Morton L. Curtis and Roger Lyndon as co-discoverers of this characterization, the result is known as Curtis--Hedlund--Lyndon theorem; it remains true in the framework of shift spaces over finite alphabets: 

\begin{thm}[Curtis--Hedlund--Lyndon Theorem {\cite[Theorem 6.2.9]{lind}}]\label{chl1}
Let $X$ and $Y$ be shift spaces over finite alphabets. A map $\Psi:X\to Y$ is a sliding block code if and only if it is continuous and shift-commuting.   
\end{thm}

This important theorem fails when the shift space is not compact, take for example the morphism $\Phi:\N^\mathbb{Z}\to\N^\mathbb{Z}$ ($\N=\{0,1,2,\cdots\}$) given by
\begin{equation}\label{ejem1}
\Phi(x)_n=\sum_{|j|\leq x_n}x_{j+n},\text{ for all $x\in X$ and every $n\in\mathbb{Z}$};
\end{equation}
it is continuous and shift-commuting but it cannot be expressed through a local rule as in \eqref{e1}. 
In the study of shift spaces and their morphisms over finite alphabets the compactness has a special role;  particularly, it allows to prove the following two results which are part, together Curtis-Hedlund-Lyndon theorem, of the folklore of symbolic dynamics and coding theory:

\begin{thm}[{\cite[Theorem 1.5.13]{lind}}]\label{image}
If $\Phi:X\to \mathcal{U}^\mathbb{Z}$ is a shift morphism, then $\Phi(X)$ is a closed subset of $\mathcal{U}^{\mathbb{Z}}$; so, it is a shift space.
\end{thm}

\begin{thm}[{\cite[Theorem 1.5.14]{lind}}]\label{rever}
If $\Phi:X\to \mathcal{U}^\mathbb{Z}$ is an injective shift morphism, then the inverse map $\Phi^{-1}:\Phi(X)\to X$ is a sliding block code.
\end{thm}

These two results are known as {\em closed image property} and {\em reversibility}, respectively. Like Curtis--Hedlund--Lyndon Theorem, both Theorem \ref{image} and Theorem \ref{rever} are not true in the non-compact context. We refer to \cite[Example 1.10.3]{cecche} and \cite[Lemma 5.1]{cecche2} where notable examples of bijective and non-reversible shift morphisms over non-finite alphabets are shown. For its part, the following example shows a non-surjective shift morphism on $\mathbb{N}^\mathbb{Z}$ whose image set is not a shift space.

\begin{exa}\label{noimage}
Let $\Psi:\N^\mathbb{Z}\to\N^\mathbb{Z}$ be the ESBC given by
$$
\Psi(x)_j=x_{j-x_j}+x_{j+x_j},\,\text{ for all $x\in\N^\mathbb{Z}$ and every $j\in\mathbb{Z}$}.
$$
By direct verification it is shown that this map is a shift morphism. It is also easy to see that the constant sequence $(x_n)_{n\in\mathbb{Z}}=1^\mathbb{Z}$ ($x_n=1$ for all $n$) has no preimage under $\Psi$.
Now, for every integer $k\geq 1$, we consider 
$x^k=(x_j^k)_{j\in\mathbb{Z}}\in\mathbb{N}^\mathbb{Z}$ where
$$
x_j^k=
\begin{cases}
3k+1-j,\,\text{ if $-k\leq j\leq k$}\\
0,\,\text{ if $j\geq k+1$}\\
1,\,\text{ if $j\leq -k-1$}
\end{cases}.
$$
By taking the $\Psi$-image of each $x^k$ one obtains that $\Psi(x^k)|_{[-k,k]}=1^{2k+1}$
for all $k\geq 1$, therefore
$\Psi(x^k)\to 1^\mathbb{Z}$ when $k\to+\infty$, this implies the image set
$\Psi(\N^\mathbb{Z})$ is not a shift space; also observe that $(x^k)_{k\geq 1}$ has no 
convergent subsequences.
\end{exa}

Each one of the results in the trilogy: Theorem \ref{chl1}, Theorem \ref{image} and Theorem \ref{rever}  
have been validated in the world of cellular automata by omitting the finiteness of the alphabet. 
For our aim it is necessary to make some brief comments concerning it; let us begin by the last two theorems. According to our knowledge of the literature on the subject, the most recent developments have been reported by Ceccherini-Silberstein and Coornaert in \cite{cecche3} and \cite{cecche2}. In these articles the authors consider as alphabet a vector space $V$ and a group $G$ substituting $\mathbb{Z}$, the vector space is endowed with the discrete topology and the cartesian product $V^G$ is equipped with the product topology; clearly $V^G$ has the natural vectorial structure induced by that of $V$. Then the concept of cellular automaton on $V^G$ is introduced in an analogous way than in the classical case, see \cite[Chapter 1]{cecche} for details. In this setting are proved the following results: 

\begin{thm}[{\cite[Theorem 1.2, Corollary 1.6 and Corollary 1.7]{cecche3}}]\label{cecchethm}
Let $V$ be a finite-dimensional space and $G$ a group. 
\begin{enumerate}[a)]
\item 
If $\tau:V^G\to V^G$ is a linear cellular automaton, then $\tau(V^G)$ is closed in $V^G$; that is, $\tau$ has the closed image property.
\item
If $\tau:V^G\to V^G$ is a bijective linear cellular automaton, then $\tau^{-1}$ is also a cellular automaton, i.e. $\tau$ is reversible.
\end{enumerate}
\end{thm}

New proofs of a) and b) are presented in \cite{cecche2} by using the technical tool of Mittag-Leffler lemma for projective sequence of sets. 
In this same article it is mentioned that more general results than a) were obtained by Gromov in \cite{gromov}. Also there it is proved that when the vector space $V$ is infinite-dimensional, it is possible to construct on $V^\mathbb{Z}$ a linear cellular automaton without the closed imagen property (see \cite[Theorem 1.5]{cecche2}) and a non-reversible bijective linear cellular automaton (see \cite[Theorem 1.2]{cecche2}). 

\smallskip

Related to the first result of the trilogy we can say that in \cite{rrv}
is proved a generalization of the Curtis--Hedlund--Lyndon theorem for cellular automata over arbitrary discrete alphabets, it is based on the concept of barrier and extended notions of local rule and sliding block code also introduced in \cite{rrv}. Such generalization is transferred without any difficulty to continuous and shift-commuting mappings between shift spaces over such kind of alphabets, it is part of the foundation on which the discussion in the present paper is developed; see Theorem \ref{chl4} in the next section for its precise statement and concise commentaries about its proof. There are several references where other different versions of the pioneer Curtis--Hedlund--Lyndon theorem are presented, for example: \cite{cecche}, 
\cite{sobottaka0}, \cite{sobottaka}, \cite{muller}, \cite{ott}, \cite{richarson}, \cite{sobottaka1} and \cite{wacker}. In some of this versions the topological structure (product topology) or the notions of shift space and sliding block code are modified, we maintain the original relative product topology on the shift spaces and the essence of the notion of local rule in order to 
preserve the continuity and shift-commuting property as the underlying concepts 
in the notion of sliding block code. 

\smallskip

Having as conceptual basis the characterization of the continuous and shift-commuting maps between shift spaces, i.e. the extended Curtis--Hedlund--Lyndon Theorem established in Theorem \ref{chl4} below, 
the main goal in this paper is to give sufficient conditions in order to guarantee the closed image property and reversibility for such mappings.

\smallskip

Since the sufficient conditions that we will establish for our purposes are somewhat technical, 
it is necessary to clarify this technicality for a better understanding of the discussion, which even delays the precise establishment of the central results of this work.
We have organized the rest of the article as follows. 
In Section \ref{sec2} we begin by reviewing the concept of barrier on shift spaces and the extended notions of local rule and sliding block code; these were introduced in \cite{rrv} as the platform to extend the classic Curtis--Hedlund--Lyndon theorem to cellular automata on arbitrary discrete alphabets (see \cite[Theorem 4]{rrv}), its translation to continuous and shift-commutating maps between shifts spaces is Theorem \ref{chl4}. Then two special types of barriers are introduced, they are linked to the maps characterized in Theorem \ref{chl4} and determine a wide subset of them on which sufficient conditions will be established for the fulfillment of the closed image property and reversibility. Additionally a crucial lemma is proved. Only after these preliminaries the main results of the article are established and proved: Theorem \ref{thm1}, Theorem \ref{thm2} and Theorem \ref{thm3}, this is precisely the content of the third and final section.

\medskip
\noindent 
{\bf Acknowledgements.}
The authors would like to thank the referee for his/her helpful comments and suggestions which led to the improvement of the manuscript. 
%%%%%%%%%%%%%%%%%%%%%%%%%%%%%%%%%%%%%%%%%%%%%%%%%%%%
\section{The extended environment and technical preliminaries}\label{sec2}
First we recall that the product topology on the full shift $\mathcal{A}^{\mathbb{Z}}$ has as a basis the set of cylinders $C(h)=\{x\in\Ai^\mathbb{Z}: x\vert_{dom(h)}=h\}$, where $h$ is a $\Ai$-valued function with 
domain $dom(h)$ a finite subset of $\mathbb{Z}$ and $x\vert_{dom(h)}$ is the 
restriction of $x:\mathbb{Z}\to\Ai$ to $dom(h)$; we denote by $F(\mathcal{A})$ the set of such functions $h$ and by abuse we also say that $dom(h)$ is the domain of the cylinder $C(h)$. 
 
\begin{defi}
Let $X\subseteq\mathcal{A}^\mathbb{Z}$ be a shift space, a {\em barrier} on $X$ is 
a partition of $X$ whose elements $C_X(h)$ are 
of the form $X\cap C(h)$, they are called {\em cylinders} in $X$.
\end{defi}

This concept was inspired by the notion of barrier introduced by Nash-Williams \cite{nash} in his 
study on well-quasi-ordered sets; for this subject see also \cite{dipris} and
\cite{frai}. 

\smallskip

Some helpful observations about barriers on a shift space $X\subseteq\mathcal{A}^\mathbb{Z}$:

\smallskip
\noindent
a)
The set of barriers on $X$ is partially ordered by the {\em finer-than} relation: a barrier $\mathcal{B}'$ is {\em finer than} the barrier $\mathcal{B}$ if for every $B'\in\mathcal{B}'$ there is $B\in\mathcal{B}$ such that $B'\subseteq B$. When $\mathcal{B}'$ is finer than $\mathcal{B}$ it is also said that $\mathcal{B}'$ is a {\em refinement} of $\mathcal{B}$ or that $\mathcal{B}$ is coarse than $\mathcal{B}'$.
Given a barrier $\Bi$, it is always possible to get a finer one; to see that, observe that if 
$C_X(h)$ is a cylinder in $X$ and $\ell\notin dom(h)$, then $C_X(h)$ is the 
disjoint union of cylinders $C_X(h_a)$, where $h_a$ is the extension of $h$ 
to $dom(h)\cup\{\ell\}$ with $h_a(\ell)=a$ going through the set of 
symbols such that $C(h_a)\cap X\neq \emptyset$. In particular, if $\Bi$ is a 
barrier on $X$ such that for some interval $[i,j]\subset\mathbb{Z}$ the inclusion
$dom(h)\subset [i,j]$ holds for every $C_X(h)\in\Bi$, then there is a 
refinement $\Bi'$ of $\Bi$ such that $[i,j]$ is the domain of each cylinder in $\Bi'$.

\smallskip
\noindent
b)
Take integers $m,n,N$ with $N=m+n+1\geq 1$, the set of all allowed $N$-blocks $\mathcal{L}_N(X)$ can be identified with a barrier whose cylinders have the same domain $[-m,n]$; in fact, every allowed $N$-block 
$w=w_0\cdots w_{N-1}$ in 
$X$ is identified with the cylinder $C_X(h_w)$ where $h_w:[-m,n]\to\mathcal{A}$ 
and $h_w(-m+\ell)=w_\ell$ for all $0\leq \ell\leq N-1$.

\begin{defi}
Let $\mathcal{A}$ and $\mathcal{U}$ be alphabets.
An {\em $\mathcal{U}$-extended local rule} in the shift space $X\subseteq\mathcal{A}^\mathbb{Z}$ (or 
simply local rule in $X$) is any 
$\mathcal{U}$-valued function whose domain is a barrier on $X$.
Given a local rule $\varphi:\Bi\to\mathcal{U}$ in $X$, the {\em map induced by 
$\varphi$} is the transformation $\Psi:X\to\mathcal{U}^\mathbb{Z}$ 
defined by
\begin{equation}\label{extended1}
\Psi(x)_n=\varphi(C_X(h_{x,n})) \text{ for every $x\in X$ and $n\in\mathbb{Z}$},
\end{equation}
where $C_X(h_{x,n})$ is the cylinder in $\Bi$ containing $\sigma^n(x)$, 
$\sigma$ is the shift map on $\mathcal{A}^\mathbb{Z}$.
\end{defi}

We emphasize that if $\mathcal{B}$ is the barrier described in paragraph b) above, then the action in   
\eqref{extended1} matches that expressed in \eqref{e1}.
So, the map induced by an extended local rule extends the classical 
notion of sliding block code. From now on we use the term {\em extended sliding 
block code} (ESBC for short) for every transformation $\Psi$ as defined in \eqref{extended1}. 
In this extended setting the classical Curtis-Hedlund-Lyndon theorem \cite[Theorem 6.2.9]{lind} 
remains valid:

\begin{thm}[Curtis-Hedlund-Lyndon theorem]\label{chl4}
Let $\mathcal{A}$ and $\mathcal{U}$ be arbitrary discrete alphabets and $X\subseteq\mathcal{A}^\mathbb{Z}$ a shift space. A map 
from $X$ to $\mathcal{U}^{\mathbb{Z}}$ is a shift morphism if and only if it is an ESBC. 
\end{thm}

A proof of this theorem is easily obtained by paraphrasing  the proof of Theorem 4 in \cite{rrv}, which characterize the cellular automata on arbitrary discrete alphabets. We only highlight the following facts:

\smallskip
\noindent
$\bullet$
If $\Psi:X\to\mathcal{U}^\mathbb{Z}$ is the ESBC induced by $\varphi:\mathcal{B}\to\mathcal{U}$,
then it is a shift morphism. In fact, it is enough to observe that for all $n\in\mathbb{Z}$ the $n$-coordinate function $\Psi_n:X\to\mathcal{U}$ of $\Psi$ satisfies $\Psi_n=\Psi_0\circ \sigma^n$, and the $0$-coordinate function is constant in each cylinder of $\mathcal{B}$.

\smallskip
\noindent
$\bullet$
Let $\Psi:X\to\mathcal{U}^\mathbb{Z}$ be a morphism. Since in $X$ any nonempty open set is 
disjoint union of cylinders in $X$, it follows that the set of preimages 
$\Psi_0^{-1}(\{b\})$, with $b$ varying in $\Psi_0(X)$, determines a barrier on $X$ and 
$\Psi_0$ defines a local rule inducing $\Psi$; obviously $\Psi_0$ is constant 
in each cylinder of such a barrier. 

\begin{exa}\label{aux2}
Only to illustrate we return to 
the morphism $\Psi:\N^\mathbb{Z}\to\N^\mathbb{Z}$ defined by
\eqref{ejem1}. It is clear that the $0$-coordinate of 
$\Psi$ is given by 
$\Psi_0(x)=\displaystyle\sum_{|j|\leq x_0}x_j$ for all 
$x=(x_n)_{n\in\mathbb{Z}}\in\N^\mathbb{Z}$. Hence, $\Psi_0^{-1}(\{m\})=\left\{x\in\N^\mathbb{Z}: \sum_{|j|\leq x_0} x_j=m\right\}$ for all $m\in\mathbb{N}$.
In other words, if $w$ is a block with odd length and 
$C_w$ denotes the cylinder of all $x\in\N^\mathbb{Z}$ such that 
$x\vert_{[-k,k]}=w$, then the preimage $\Psi_0^{-1}(\{m\})$ is the disjoint 
union of the cylinders $C_w$, where $w=w_{-w_0}\cdots w_{0}\cdots w_{w_0}$ and 
$\displaystyle\sum_{|j|\leq w_0}w_j=m$. Therefore the collection $\Bi^0$ of cylinders 
$C_w$ with $m$ runs $\N$ is a barrier on $\N^\mathbb{Z}$ and the local rule 
$\varphi:\Bi^0\to\N$, with 
$\varphi(C_w)=\displaystyle\sum_{|j|\leq w_0}w_j$, induces $\Psi$. Note that $\Psi_0^{-1}(\{0\})=C_0=\{x\in\N^\mathbb{Z}: x_0=0\}$ and it is the union of the cylinders $C(h_v)$, where $v\in\mathbb{N}$ and $h_v:\{0,1\}\to\mathbb{N}$ is given by $h_v(0)=0$ and $h_v(1)=v$. 
If in $\mathcal{B}^0$ one substitutes  
$C_0$ by the collection of these new
cylinders $C(h_v)$, a finer barrier $\Bi^1$ is obtained and $\varphi$ acting on 
$\Bi^1$ also induces $\Psi$; clearly $\Psi_0(C(h_v))=0$ for all $v\in\mathbb{N}$. We also observe that $\Bi^0$ is the maximum of all the barriers on $\N^\mathbb{Z}$ inducing $\Psi$, this follows from the following fact: if $C(h)$ is a cylinder on $\mathbb{N}$ such that $\Psi_0$ is constant on $C(h)$, then there exist an $\mathbb{Z}$-interval $[-N,N]$ and a block $w\in\mathbb{N}^{2N+1}$ such that $[-N,N]\subseteq dom(h)$ and $h|_{[-N,N]}=w$.
\end{exa}

We now introduce two type of barriers and a particular class of sequences related to an ESBC, these notions are fundamentals for the rest of the article.

\begin{defi}
Let $\Psi:X\to\mathcal{U}^\mathbb{Z}$ be an ESBC.
\begin{enumerate}[i)]
\item 
A barrier $\Bi$ on $X$ is said to be {\em attached to} $\Psi$ if the $0$-coordinate function of $\Psi$ is constant 
on every cylinder of $\Bi$.
\item
If there exists a barrier $\mathcal{B}$ attached to $\Psi$ 
such that for every $\ell\in\Psi_0(X)$ it is finite the number of cylinders in $\mathcal{B}$ with $\Psi_0(B)=\ell$, then it is said both $\Psi$ and $\mathcal{B}$ are of {\em finite degree}.
\item
A sequence $(x^k)_{k\in\mathbb{Z}}\subset X$ is called {\em distinguished} for $\Psi$ if $(\Psi(x^k))_{k\in\mathbb{Z}}$ is convergent. The set of distinguished sequence for $\Psi$ is denoted by $X(\Psi)$. 
\end{enumerate}
\end{defi}

It is easy to see that a barrier $\Bi$ on $X$ is attached to an ESBC 
$\Psi:X\to\mathcal{U}^\mathbb{Z}$ if and only if it is induced by the local rule $\varphi:\Bi\to\mathcal{U}$ defined, 
for each $C\in\Bi$, by $\varphi(C)=\Psi_0(C)$. 
In this way, the barriers attached to an ESBC are the domains of the local rules inducing it. 

\smallskip

When the alphabets are finite every sliding block code is of finite degree.  
Not every barrier attached to a finite degree ESBC is finite degree, such is the case of the barriers $\mathcal{B}^0$ and $\mathcal{B}^1$ described in Example \ref{aux2}. On the other hand, not every ESBC is 
of finite degree, this is shown in the following example; it also shows that there exist distinguished sequences without convergent subsequences. 

\begin{exa}\label{nofinitedegree}
Let $\Psi:\N^\mathbb{Z}\to\N^\mathbb{Z}$ be the ESBC given by
$$
\Psi(x)_j=x_{j-x_j}+x_{j+x_j},\,\text{ for all $x\in\N^\mathbb{Z}$ and every $j\in\mathbb{Z}$}.
$$
Let us see that $\Psi$ is not of finite degree. First, take $\ell,m,a,b\in\mathbb{N}$ such that $a+b=\ell$; for this collection of natural numbers define $h:\{-m,0,m\}\to\mathbb{N}$ by $h(-m)=a,h(0)=m$ and $h(m)=b$. It is clear that for every $x\in C(h)$ one has $\Psi_0(x)=\ell$, so the set of cylinders thus defined constitutes a barrier $\mathcal{B}$ attached to $\Psi$; also observe that $\mathcal{B}$ is not of finite degree. Now take any cylinder $C(h')$ such that for some $\ell\in\mathbb{N}$ the equality $\Psi_0(x)=\ell$ holds for each $x\in C(h')$. Suppose that $0\notin dom(h')$ and pick $m\notin dom(h')\cup\{0\}$, so the $m$-coordinate of $x\in C(h')$ with $x_0=m$ is restricted to $x_{-m}+x_m=\ell$; this contradicts the fact that the $m$-coordinate of points in $C(h')$ varies throughout $\mathbb{N}$ even when the $0$-coordinate is $m$. Thus $0$ necessarily belongs to $dom(h')$ and so $x_{-h'(0)}+x_{h'(0)}=\ell$ for all $x\in C(h')$, which forces that $-h'(0),h'(0)\in dom(h')$. Consequently $C(h')$ is contained in some cylinder $C(h)$ as described above; therefore every barrier attached to $\Psi$ is a refinement of $\mathcal{B}$, which clearly implies that $\Psi$ is not of finite degree.

Next, it is not hard to see that $(x_n)_{n\in\mathbb{Z}}=1^\mathbb{Z}$ (i.e. $x_n=1$ 
for all $n\in\mathbb{Z}$) has no preimage under $\Psi$. Now, for every integer $k\geq 1$ consider 
$x^k=(x_j^k)_{j\in\mathbb{Z}}\in\mathbb{N}^\mathbb{Z}$ defined for each $j\in\mathbb{Z}$ as follows:
$$
x_j^k=
\begin{cases}
3k+1-j,\,\text{ if $-k\leq j\leq k$}\\
0,\,\text{ if $j\geq k+1$}\\
1,\,\text{ if $j\leq -k-1$}
\end{cases}.
$$
After a direct calculation one has
$\Psi(x^k)\vert_{[-k,k]}=1^{2k+1}$ for all $k\geq 0$, therefore
$\Psi(x^k)\to 1^\mathbb{Z}$ when $k\to+\infty$; this shows that the sequence 
$(x^k)_{k\geq 1}$ is distinguished for $\Psi$ but it has no 
convergent subsequences.   
\end{exa}

In what follows, a sequence of elements in a shift space will be called 
{\em nice} if it has convergent subsequences. 
An immediate fact is the following:

\smallskip
\noindent
{\em If $\Psi:X\to\mathcal{U}^\mathbb{Z}$ is an ESBC and every distinguished sequence by $\Psi$ is  
nice, then $\Psi$ has the closed imagen property; i.e. $\Psi(X)$ is a shift space}.

\smallskip

We accentuate that 
the converse of this assertion is false, this is shown in Example \ref{exam2} below.
By restricting our attention to ESBC of finite degree, we will give sufficient conditions more technical than the totality of distinguished nice sequences for an ESBC to have the closed image property. Furthermore, these same conditions will guarantee the reversibility of bijective ESBC. For this purpose the following lemma is crucial.
Before we need to introduce some notations to make the sentences more readable. First we recall that $F(\mathcal{A})$ is the set of all $\mathcal{A}$-valued functions whose domains are finite subsets of $\mathbb{Z}$. 
Let $X\subseteq\mathcal{A}^\mathbb{Z}$ be a shift space and $\sigma$ the shift map on $\mathcal{A}^\mathbb{Z}$. Observe that
for any $n\in\mathbb{Z}$ and every cylinder $C(h)$ in $\mathcal{A}^\mathbb{Z}$, $\sigma^{-n}(C(h))$ is also 
a cylinder, we denote it by $C(h^{[n]})$ where $h^{[n]}$ is the $n$-translation of $h$, that is:
$$
dom(h^{[n]})=dom(h)+n\,\text{ and $h^{[n]}(j+n)=h(j)$ for all $j\in dom(h)$}.
$$
On the other hand, it is easy to see that
if $\{h_\alpha\}_{\alpha\in\Gamma}$ is a collection in $F(\mathcal{A})$ satisfying  
$\bigcap_{\alpha\in\Gamma} C(h_\alpha)\neq\emptyset$, then $h_\alpha(m)=h_\beta(m)$ holds for all $\alpha,\beta\in\Gamma$ and every integer $m$ in $dom(h_\alpha)\cap dom(h_\beta)$. 
It allows to define $\bigoplus_{\alpha\in\Gamma} 
h_\alpha: \bigcup_{\alpha\in\Gamma} dom(h_\alpha)\to\mathcal{A}$ by
$$
\left(\bigoplus\nolimits_{\alpha\in\Gamma}h_\alpha\right)(m)=h_\alpha(m)\,\text
{ whenever $m\in dom(h_\alpha)$}.
$$
Obviously this new function belongs to $F(\mathcal{A})$ if and only if 
$\bigcup_{\alpha\in\Gamma} dom(h_\alpha)$ is finite; in this case 
$\bigcap_{\alpha\in\Gamma} C(h_\alpha)$ is just the 
cylinder $C(\bigoplus_{\alpha\in\Gamma}h_\alpha)$. 
A particular case of this kind of collection is a sequence of functions $h_n$ in $F(\mathcal{A})$ such that $h_{n+1}$ extends $h_n$.
Anyway, it is clear that for any collection $\{h_\alpha\}_{\alpha\in\Gamma}\subset F(\mathcal{A})$
$$
\bigcap\nolimits_{\alpha\in\Gamma} C(h_\alpha)=\{x\in\mathcal{A}^\mathbb{Z}:\text{ for all 
$m\in\mathbb{Z}$, $x_m=h_\alpha(m)$ if $m\in dom(h_\alpha)$}\}.
$$

An additional notation, let $\Psi:X\to\mathcal{U}^\mathbb{Z}$ be an ESBC and $\mathcal{B}$ a barrier attached to it. For each symbol $\ell\in\Psi_0(X)$, $\mathcal{B}(\ell)$ denotes the set of cylinders $B$ in $\mathcal{B}$ such that $\Psi_0(B)=\ell$. Obviously $\mathcal{B}$ is of finite degree if the cardinal of $\mathcal{B}(\ell)$ is finite for all $\ell\in\Psi_0(X)$.

\begin{lemma}[Crucial lemma]\label{aux3}
If $\Psi:X\to\mathcal{U}^\mathbb{Z}$ is an ESBC of finite 
degree, then for all integer $\ell\geq 0$ and every $(x^k)_{k\geq 1}\in X(\Psi)$ with $\Psi(x^k)\to y$ 
when $k\to+\infty$, there exist 
an infinite subset $S_\ell$ of $\mathbb{N}$ and a cylinder 
$C_X(h_\ell)$ such that:
$S_{\ell+1}\subset S_\ell$,
$C_X(h_{\ell+1})\subset C_X(h_{\ell})$,
$x^k\in C_X(h_{\ell})$ for all $k\in S_\ell$ and 
$\Psi(z)\vert_{[-\ell,\ell]}=y\vert_{[-\ell,\ell]}$ for every 
$z\in C_X(h_{\ell})$.
\end{lemma}
\begin{proof}
Take a barrier $\Bi$ attached to $\Psi$ of finite degree. Let $(x^k)_{k\geq 1}$ be a sequence in $X(\Psi)$ with $\Psi(x^k)\to y$ if $k\to+\infty$.  
From the definition of the Cantor metric it is clear that
for each natural number $\ell\geq 0$ there is 
$N_\ell\geq 1$ such that 
$x^k\in\Psi_j^{-1}(y_j)$ for all $|j|\leq\ell$ and every $k\geq N_\ell$, where $y=(y_j)_{j\in\mathbb{Z}}$. In 
particular, for $\ell=0$ one can 
choose a cylinder $C_X(h_{y_0,0})\in\Bi(y_0)$ and an infinite subset $S_0$ of 
$\N$ such that $x^k\in C_X(h_{y_0,0})$ for all $k\in S_0$ and $\Psi_0(z)=y_0$ 
for each $z\in C_X(h_{y_0,0})$. Further, for $\ell=1$ there are cylinders 
$C_X(h_{y_{-1},-1})\in\Bi(y_{-1})$ and $C_X(h_{y_{1},1})\in\Bi(y_{1})$ and there exists 
an infinite set $S_1\subset S_0$ in such a way that
$$
x^k\in C_X(h_{y_1,1}^{[1]})\cap C_X(h_{y_0,0})\cap 
C_X(h_{y_{-1},-1}^{[-1]})\,\text{ for all $k\in S_1$}.
$$
Observe that this intersection is the cylinder $C_X(h_1)$ where 
$h_1=\bigoplus_{|j|\leq 1}h_{y_j,j}^{[j]}$ and $h_{y_0,0}^{[0]}=h_{y_0,0}$; also note that 
$\Psi(z)\vert_{[-1,1]}=y\vert_{[-1,1]}$ for all $z\in C_X(h_1)$. 
Proceeding by recurrence, in each step $\ell\geq 1$ one can select 
an infinite set $S_\ell\subset S_{\ell-1}$ and cylinders
$C_X(h_{y_{-\ell},-\ell})\in\Bi(y_{-\ell})$ and  $C_X(h_{y_{\ell},\ell})\in\Bi(y_{\ell})$ such that 
$$
x^k\in\bigcap_{|j|\leq\ell}C_X(h_{y_j,j}^{[j]}),\,\text{ for all $k\in 
S_\ell$}. 
$$
The proof finishes by making $h_\ell=\bigoplus_{|j|\leq \ell}h_{y_j,j}^{[j]}$ 
and observing that $\Psi(z)\vert_{[-\ell,\ell]}$ matches $y\vert_{[-\ell,\ell]}$ for all $z$ in the cylinder
$C_X(h_\ell)=\bigcap_{|j|\leq\ell}C_X(h_{y_j,j}^{[j]})$. 
Notice that the sequence of cylinders $C_X(h_{\ell})$ is decreasing because 
the function $h_{\ell+1}$ extends $h_\ell$ for every $\ell\geq 0$ and 
$h_0=h_{y_0,0}$. 
\end{proof}

\begin{rk}\label{r22}
Let $\Psi:X\to\mathcal{U}^\mathbb{Z}$ be an ESBC and $\mathcal{B}$ a barrier attached to $\Psi$ of finite degree. With the notations of the preceding lemma we observe that if 
$(x^k)_{k\geq 1}$ is a sequence in $X(\Psi)$ and $\ell$ is a natural number, then both the 
function $h_\ell$ and the set $S_\ell$ may be not unique; they depend on the barrier and the limit of the sequence $(\Psi(x^k))_{k\geq 1}$.
Nonetheless, for any 
chosen sequence $(h_\ell)_{\ell\geq 0}$ one has that 
$\bigcap_{\ell\geq 0}C(h_\ell)\neq\emptyset$, and so the function
$h_\infty:=\bigoplus_{\ell\geq 0}h_\ell$ is well defined; obviously 
$x$ belongs to $\bigcap_{\ell\geq 0}C(h_\ell)$ if, and only if,
$x\vert_{dom(h_\infty)}=h_\infty$. It is also clear that although 
$C_X(h_\ell)\neq\emptyset$ for all $\ell\geq 0$, it may happen that 
$\bigcap_{\ell\geq 0}C_X(h_\ell)=\emptyset$. On the other hand, 
if $\bigcap_{\ell\geq 0}C_X(h_\ell)\neq\emptyset$, every point of this set  
is mapped by $\Psi$ on the limit of the sequence $(\Psi(x^k))_{k\geq 1}$. 
Clearly in this case
$$
\bigcap_{\ell\geq 0}C_X(h_\ell)=\{x\in X: x\vert_{dom(h_\infty)}=h_\infty\}.
$$

After this discussion, we say that an $\mathcal{A}$-valued function $h_\infty$ is {\em associated} to a sequence in $X(\Psi)$ and a barrier attached to $\Psi$ of finite degree whenever it is obtained as described above.  
Although this association may be multivalued: different functions $h_\infty$ may be associated to a same distinguished sequence or barrier, all of them are 
classified into five exclusive and exhaustive classes, according to how their domains are: 

\begin{enumerate}[$C1$.]
\item
$dom(h_\infty)$ is bounded. 

\item
$dom(h_\infty)=\mathbb{Z}$.

\item
$dom(h_\infty)$ is bilaterally unbounded, that is
$dom(h_\infty)\varsubsetneq\mathbb{Z}$ and for all $N>0$ 
there are $a,b\in dom(h_\infty)$ such that $a<-N$ and $b>N$.

\item
$dom(h_\infty)$ is left-unbounded, that means that there is 
$M\in\mathbb{Z}$ such that $(-\infty,M]$ is the minimal interval 
containing $dom(h_\infty)$.

\item
$dom(h_\infty)$ is right-unbounded: there is 
$m\in\mathbb{Z}$ such that $[m,+\infty)$ is the minimal interval 
containing $dom(h_\infty)$.
\end{enumerate}
\end{rk}
%%%%%%%%%%%%%%%%%%%%%%%%%%%%%%%%%%%%%%%%%%%%%%%%%%%%%%%%%%%%%%%%
\section{Statements and proofs of the main results}
Only now are we able to state and demonstrate the main results of this article, that is the goal of this final section.

\subsection{On the closed image property of ESBC}
Let $X\subseteq\mathcal{A}^\mathbb{Z}$ be a shift space and $\Psi:X\to\mathcal{U}^\mathbb{Z}$ an ESBC. Since $\Psi$ is shift-commuting, it is clear that the image set $\Psi(X)$ is a 
shift space if and only if $\Psi(X)$ is a closed subset 
$\mathcal{U}^\mathbb{Z}$; that is, the limit of the image of every distinguished sequence for $\Psi$ belongs to $\Psi(X)$.

\smallskip

Our first main result is related to the classes $C1$ and $C2$.

\begin{thm}\label{thm1}
Let $\Psi:X\to\mathcal{U}^\mathbb{Z}$ an ESBC of finite degree. If a function $h_\infty$ in $C1\cup C2$ is associated to each sequence in $X(\Psi)$, then $\Psi(X)$ is a 
shift space. 
\end{thm}
\begin{proof}
Take any point $y$ in the closure of $\Psi(X)$. Let 
$(y^k)_{k\geq 1}$ be a sequence in $\Psi(X)$ such that $y^k\to y$ when 
$k\to+\infty$.
Pick $(x^k)_{k\geq 1}\subset X$ with $\Psi(x^k)=y^k$ for every 
$k\geq 1$; obviously $(x^k)_{k\geq 1}$ belongs to $X(\Psi)$. 
Consider a barrier attached to $\Psi$ of finite degree in such a way that 
for sequences $(h_\ell)_{\ell\geq 0}$, $(S_\ell)_{\ell\geq 0}$ as in Lemma 
\ref{aux3}, the corresponding function $h_\infty\in C1\cup C2$. First we
assume that $h_\infty$ belongs to the class $C2$, that is $dom(h_\infty)=\mathbb{Z}$; from this same lemma,  
strictly increasing sequences of integers $(\ell_M)_{M\geq 1}$ and 
$(k_M)_{M\geq 1}$ can be selected such that $[-M,M]\subset dom(h_{\ell_M})$, 
$k_M\in S_{\ell_M}$ and $x^{k_M}\to h_\infty$ when $M\to+\infty$. So, the 
continuity of $\Psi$ implies that $y\in\Psi(X)$. Now we suppose that
$h_\infty$ belongs to the class $C1$. As $h_{\ell+1}$ extends 
$h_\ell$ for all $\ell\geq 0$ and 
$dom(h_\infty)=\bigcup_{\ell\geq 0}dom(h_\ell)$, there is $L\geq 0$ such that 
$h_\ell=h_L$ for all $\ell\geq L$. Thus 
$\bigcap_{\ell\geq 0}C_X(h_\ell)$ is the cylinder $C_X(h_L)$, again from 
Lemma \ref{aux3} one deduces that
$\Psi(z)=y$ for all $z\in C_X(h_L)$.
\end{proof}

\begin{rk}\label{rk3}
From the proof of the preceding theorem it is clear that if $h_\infty\in C2$, 
then the sequence $(x^k)_{k\geq 1}$ in $X(\Psi)$ inducing it is nice; however, 
the same cannot be said if $h_\infty\in C1$ (see Example \ref{exam2} below); 
even so, in this case, the limit of $(\Psi(x^k))_{k\geq 1}$ belongs to 
$\Psi(X)$. We notice that if $\Psi$ is injective and $h_\infty\in C1$, then 
$(x^k)_{k\geq 1}$ is nice; indeed, it has a constant subsequence.

\smallskip

We also would like to highlight that 
if $\Psi:X\to\mathcal{U}^\mathbb{Z}$ is an ESBC and the domains of the cylinders in a barrier $\Bi$ attached to $\Psi$  have a common point, then every function $h_\infty$ associated to any sequence in $X(\Psi)$ belongs to the class $C2$. This 
follows from the next fact: 
if $m\in dom(h)$ for all $h$ with $C_X(h)\in\Bi$, then for all 
$\ell\geq 0$ the domain of $h_\ell$ contains 
the $\mathbb{Z}$-interval $[-\ell+m,m+\ell]$. 
\end{rk}

\begin{exa}\label{exam2}
Let $X$ denote the shift space of all sequences $x=(x_n)_{n\in\mathbb{Z}}$ in $\N^\mathbb{Z}$ 
such that $0$ appears in $x$ and $x_n\neq x_m$ for all $n\neq m$. For every 
$x\in X$, let $0(x)$ be the integer where $0$ appears in $x$. Define 
$\Psi:X\to\mathbb{Z}^\mathbb{Z}$ 
by $\Psi(x)_n=0(x)-n$ for all $n\in\mathbb{Z}$ and every $x\in X$. It is easy to check that $\Psi$ is a 
shift morphism. Observe that if
$C_X(f_n)$ is the cylinder of all $x\in X$ such 
that $0(x)=n$, then 
$\mathcal{B}=\{C_X(f_n):n\in\mathbb{Z}\}$ is a barrier attached to $\Psi$ of finite degree; indeed, for every $n\in\mathbb{Z}$ one has 
$\Psi_0^{-1}(n)=C_X(f_n)$ and $\Psi(C_X(f_n))$ is the singleton $\{y^n\}$, where $y^n=(y^n_m)_{m\in\mathbb{Z}}$ 
with $y^n_m=n-m$. In this way, the image set $\Psi(X)$ is the discrete set $\{y^n:n\in\mathbb{Z}\}$ which is 
clearly a shift space. 
In addition, for every $(x^k)_{k\geq 1}\in X(\Psi)$, the sequence $(\Psi(x^k))_{k\geq 1}$ is eventually constant. Therefore, if one follows the indications in Lemma \ref{aux3} and Remark \ref{r22} to construct $h_\infty$ from the barrier $\mathcal{B}$, one obtains that for every distinguished sequence 
for $\Psi$ there exists $n\in\mathbb{Z}$ such that the function $h_\infty$ has domain $\{n\}$ and $h_\infty(n)=0$. Obviously in this case the function $h_\infty$ is unique, it belongs to the class $C1$ and  
there are infinitely many no nice distinguished sequences for $\Psi$. 

\smallskip

Extra informations about this shift space $X$:
$\mathcal{L}_1(X)=\N$ and for all $a,b\in\N$ with $a\neq b$, both $ab$ and $ba$ are allowed blocks in $X$.
\end{exa}

In our search of sufficient conditions to guarantee the closed image property for 
finite degree ESBC it remains to examine the cases when the function $h_\infty$ is in some of the classes $C3, C4$ or $C5$. For this end, we only 
deal with shift spaces having certain finiteness properties on the allowed 
symbols.

\begin{defi}\label{finite}
A shift space $X\subset\mathcal{A}^\mathbb{Z}$ is said to be:
\begin{enumerate}[a)]
\item
{\em right-finite} if for all 
$a\in\mathcal{L}_1(X)$, the set $\{b\in\mathcal{A}: ab\in\mathcal{L}(X)\}$ is finite.
\item
{\em left-finite} if for all 
$a\in\mathcal{L}_1(X)$, the set $\{b\in\mathcal{A}: ba\in\mathcal{L}(X)\}$ is finite. 
\item
{\em bilaterally-finite} if it is both right-finite and left-finite.
\end{enumerate}
\end{defi}

In \cite{ott} it is introduced the 
term row-finite shift with the same meaning of right-finite one; see also 
\cite{sobottaka} where is also introduced the notion of column-finite shift.  

\smallskip

For the next lemma we assume that $\Psi:X\to\mathcal{U}^\mathbb{Z}$ is an ESBC, $\mathcal{B}$ is a barrier attached to $\Psi$ of finite degree, $(x^k)_{k\geq 1}$ is a sequence in $X(\Psi)$ and $h_\infty$ is a 
function associated to $(x^k)_{k\geq 1}$ and $\mathcal{B}$.

\begin{lemma}\label{p1}
For the sequence $(x^k)_{k\geq 1}$ to be nice it 
is sufficient that one of the following conditions holds:
\begin{enumerate}[a)]
\item
$h_\infty\in C3$ and $X$ is right-finite or left-finite.
\item
$h_\infty\in C4$ and $X$ is right-finite.
\item
$h_\infty\in C5$ and $X$ is left-finite.
\end{enumerate}
\end{lemma}
\begin{proof}
$a)$ We assume $X$ right-finite, the case left-finite is treated analogously. 
Let $(h_\ell)_{\ell\geq 0}$ and $(S_\ell)_{\ell\geq 0}$ be sequences
as in Lemma \ref{aux3} such that the associated function $h_\infty\in C3$. For each 
$\ell\geq 0$ we denote by $I_\ell=[a_\ell,b_\ell]$ the 
minimal interval in $\mathbb{Z}$ such that $dom(h_\ell)\subset I_\ell$;
clearly $I_\ell\subset I_{\ell+1}$ 
for all $\ell\geq 0$ and $\bigcup_{\ell\geq 0}I_\ell=\mathbb{Z}$. Hence, there exists a 
first integer $\ell_1$ such that $h_{\ell_1}$ has gaps; that is, 
$dom(h_\ell)=I_\ell$ for all $0\leq\ell<\ell_1$ and there are integers 
$k_1\geq 1$ and 
$a_{\ell_1}\leq a_{\ell_1}^1<b_{\ell_1}^1<\cdots 
<a_{\ell_1}^{k_1}<b_{\ell_1}^{k_1}\leq b_{\ell_1}$ in such a way that 
$dom(h_{\ell_1})=I_{\ell_1}\setminus\bigcup_{i=1}^{k_1}
(a_{\ell_1}^i,b_{\ell_1}^i)$; the open intervals $(a_{\ell_1}^i,b_{\ell_1}^i)$, 
$i=1,\cdots,k_1$, are the gaps in $dom(h_{\ell_1})$. On the other hand, since 
the 
block $w_1=x^k\vert_{[a_{\ell_1},a_{\ell_1}^1]}$ is the same for all 
$k\in S_{\ell_1}$ and $X$ is right-finite, there exist an 
$(b_{\ell_1}^1-a_{\ell_1}^1-1)$-block $w$ and an infinite subset 
$S_{\ell_1}^\prime$ of $S_{\ell_1}$ such that 
$x^k\vert_{[a_{\ell_1},b_{\ell_1}^1]}=w_1w$ for all $k\in S_{\ell_1}^\prime$. 
Repeating this procedure in each gap of $dom(h_{\ell_1})$, both the 
function $h_{\ell_1}$ and the set $S_{\ell_1}$ are upgraded so that:
$dom(h_{\ell_1})=[a_{\ell_1},b_{\ell_1}]$, $S_{\ell_1}\subset S_{\ell_1-1}$ and 
$x^k\in C_X(h_{\ell_1})$ for all $k\in S_{\ell_1}$; for simplicity we have used the same notation for such update.
Thus, with recursive arguments two new sequences
$(h_\ell)_{\ell\geq 0}$ and $(S_\ell)_{\ell\geq 0}$ are constructed so that
for all $\ell\geq 0$ the following assertions hold: 
\begin{enumerate}[$\bullet$]
\item
$dom(h_\ell)=[a_\ell,b_\ell]$, $h_{\ell+1}$ extends $h_\ell$ and 
$\bigcup_{\ell\geq 0}[a_\ell,b_\ell]=\mathbb{Z}$.
\item
$S_\ell$ is an infinite subset of $\N$ with $S_{\ell+1}\subset S_\ell$ and 
$x^k\in C_X(h_\ell)$ for all $k\in S_\ell$.
\end{enumerate}
The proof of parte a) ends
by noting that the new function $h_\infty=\bigoplus_{\ell \geq 0}h_\ell$ has 
domain $\mathbb{Z}$ and therefore the sequence $(x^k)_{k\geq 1}$ is nice; see proof of Theorem \ref{thm1}.

\smallskip

The proofs for $b)$ and $c)$ essentially follow with the same 
scheme of the previous one. For instance in the case $b)$ take 
$(h_\ell)_{\ell\geq 0}$, $(S_\ell)_{\ell\geq 0}$ are as in Lemma 
\ref{aux3} and consider the corresponding funtion $h_\infty$;
if $[a_\ell,b_\ell]$ and $(-\infty,M]$ are the minimal intervals containing 
respectively $dom(h_\ell)$ and $dom(h_\infty)$, then one can assume that 
$M\in dom(h_0)$ and $h_0$ has gaps, say $(a_0^0,b_0^0),\cdots, 
(a_0^{k_0},b_0^{k_0})$. Next observe that the blocks 
$x^k\vert_{[a_{0},a_{0}^{0}]}, x^k\vert_{[b_{0}^j,a_{0}^{j+1}]}$ with 
$0\leq j< k_0$ and $x^k\vert_{[b_{0}^{k_0},M]}$ do not change when $k\in 
S_0$; so the right-finiteness property allow to select an 
$(M-a_0)$-block $w$ and an infinite set $S_0$ (same nomenclature for short) 
such that $x^k\vert_{[a_{0},M+1]}=w$ for all $k\in S_0$. Therefore, both the 
sequence of functions $(h_\ell)_{\ell\geq 0}$ and the nested sequence of 
infinite subsets $(S_\ell)_{\ell\geq 0}$ can be upgraded in such a way that for 
all $\ell\geq 0$ one has $dom(h_\ell)=[a_\ell,M+1+\ell]$ and $x^k\in 
C_X(h_\ell)$ for every $k\in S_\ell$. This leads to a new function $h_\infty$ 
whose domain is $\mathbb{Z}$ and then $(x^k)_{k\geq 1}$ is a nice sequence.
\end{proof}

\begin{rk}\label{rk2}
Observe that when $X$ is bilaterally-finite, every sequence $(x^k)_{k\geq 1}$ 
in $X(\Psi)$ is nice; in fact, if there exists a barrier attached to $\Psi$ of 
finite degree such that $h_\infty\in C1$, then the sequences 
$(h_\ell)_{\ell\geq 0}$ and $(S_\ell)_{\ell\geq 0}$ can be upgraded as above in 
such a way that the new function $h_\infty$ has $\mathbb{Z}$ as its domain.
\end{rk}

The following result complements Theorem \ref{thm1} in order to establish sufficient conditions to guarantee the closed image property for ESBC of finite degree. 

\begin{thm}\label{thm2}
Let $\Psi:X\to\mathcal{U}^\mathbb{Z}$ be an ESBC of finite degree. If one of the following 
conditions is fulfilled, then $\Psi(X)$ is shift space:
\begin{enumerate}[a)]
\item
$X$ is bilaterally-finite.
\item
$X$ is right-finite and a function $h_\infty\in C1\cup C2\cup C3\cup C4$ is associated to each sequence in $X(\Psi)$.
\item
$X$ is left-finite and a function $h_\infty\in C1\cup C2\cup C3\cup C5$ is associated to each sequence in $X(\Psi)$.
\end{enumerate}
\end{thm}
\begin{proof}
It is a straightforward combination of the arguments developed in Theorem 
\ref{thm1} (see Remark \ref{rk3}) and Lemma \ref{p1}.
\end{proof}

%%%%%%%%%%%%%%%%%%%%%%%%%%%%%%%%%%%%%%%%%%%%%%%%%%%%%%%%%%%%%%%%%%%
\subsection{On the reversibility property of ESBC}
%%%%%%%%%%%%%%%%%%%%%%%%%%%%%%%%%%%%%%%%%%%%%%%%%%%%%%%%%%%%%%%%%%%%
We begin this last part of the article by recalling that an ESBC $\Psi:X\to Y$ is said to be {\em reversible} if it is bijective and its inverse is also an ESBC.

We have already mentioned that in the case of finite alphabets, every bijective shift morphism is reversible (\cite[Theorem 1.5.14]{lind}), and this is not true when that symbol sets are not finite
(see \cite[Example 1.10.3]{cecche} and \cite[Lemma 5.1]{cecche2}). It is also well 
known that in the classical cellular automata context, reversibility is equivalent to 
injectivity, this was proved independently by Hedlund \cite{hedlund} and 
Richardson \cite{richarson}.
We consider relevant to highlight that the notion of 
reversibility has captured the interest of researchers 
in several scientific and technological scenarios, see for example 
\cite{kari}, \cite{morita1} and \cite{morita}. 

\smallskip

Our third and last main contribution is the following theorem.

\begin{thm}\label{thm3}
A bijective ESBC of finite degree is reversible if one of the 
conditions established in theorems \ref{thm1} or \ref{thm2} is satisfied.
\end{thm}
\begin{proof}
Let $X\subset \mathcal{A}^\mathbb{Z}$ and $Y\subset\mathcal{U}^\mathbb{Z}$ be shift spaces. Take a
bijective ESBC $\Psi:X\to Y$ of finite degree, we denote by $\Phi:Y\to X$ 
its inverse; clearly it is shift-commuting. So, from 
Theorem \ref{chl4} it follows that $\Phi$ is an ESBC whenever
it is continuous; in other words, $\Psi$ is reversible if the $0$-coordinate 
function $\Phi_0:Y\to \mathcal{A}$ of $\Phi$ is locally constant; recall that the 
alphabet $\mathcal{U}$ is endowed with the discrete 
topology and $\Phi_n=\Phi_0\circ \sigma^n$ for every $n\in\mathbb{Z}$, here $\sigma$ 
denotes the shift map on $\mathcal{U}^\mathbb{Z}$. 

\smallskip

Suppose 
that $\Phi_0$ is not locally constant, so for some $y\in Y$ there exists a 
sequence $(y^k)_{k\geq 1}\subset Y$ such that $y^k\to y$ when $k\to +\infty$ 
and $\Phi_0(y^k)\neq \Phi_0(y)$. Now, let $x=(x_n)_{n\in\mathbb{Z}}$ and 
$x^k=(x^k_n)_{n\in\mathbb{Z}}$ be the unique elements in $X$ such that $\Psi(x)=y$ and 
$\Psi(x^k)=y^k$ for all $k\geq 1$; obviously $(x^k)_{k\geq 1}$ is a distinguished sequence for $\Psi$ 
and $x^k_0\neq x_0$ for all $k\geq 1$. Take a barrier $\Bi$ attached 
to $\Psi$ of finite degree, let $(h_\ell)_{\ell\geq 0}$ and 
$(S_\ell)_{\ell\geq 0}$ be sequences related to $(x^k)_{k\geq 1}$ and $\Bi$ 
as in Lemma \ref{aux3} and let $h_\infty$ be an associated function to $(x^k)_{k\geq 1}$ and $\mathcal{B}$. 
Notice that if $(x^k)_{k\geq 1}$ is nice, then the injectivity of $\Psi$ leads 
to a contradiction; in fact, if $z\in X$ is a limit point of $(x^k)_{k\geq 1}$, then 
$\Psi(z)=\Psi(x)$; however, $z\neq x$ because $x^k_0\neq x_0$ for all $k$. 
The proof ends by noting that if either $h_\infty\in C1\cup C2$, or 
$h_\infty\notin C1\cup C2$ and $X$ has the appropriated finiteness property (right-finite or left-finite), then the sequence $(x^k)_{k\geq 1}$ is always nice, see Remark \ref{rk3} and proof of
Lemma \ref{p1}. 
\end{proof}

The following corollary is obvious.
\begin{cor}
If $\Psi:X\to\mathcal{U}^\mathbb{Z}$ is an injective ESBC of finite degree and one of the 
conditions established in theorems \ref{thm1} and \ref{thm2} is satisfied, then $\Psi(X)$ is a shift space and 
$\Psi:X\to\Psi(X)$ is reversible.
\end{cor}

The next and last example shows a reversible finite degree sliding block code whose inverse is neither of finite degree nor a sliding block code. 

\begin{exa}\label{arre}
Let $\Psi:\N^\mathbb{Z}\to\N^\mathbb{Z}$ the sliding block code  of memory 0 and anticipation 1 
defined for each $x=(x_n)_{n\in\mathbb{Z}}\in\N^\mathbb{Z}$ by  
$$
\Psi(x)_n=x_n+2x_{n+1}, \,\text{for all $n\in\mathbb{Z}$}.
$$
Clearly $\Psi$ is induced by the local rule $\psi:\N^2\to\N$ given 
by $\psi(uv)=u+2v$. It is not difficult to verify that $\Psi$ is injective but not onto; on the other hand, 
since for each natural number $y$ there are only finite $2$-blocks $uv$ in $\mathbb{N}^2$ solving $u+2v=y$, the sliding block code $\Psi$ is of finite degree. 
From the preceding discussions it 
follows that the image set $Y=\Psi(\N^\mathbb{Z})$ is a shift space and its 
inverse $\Phi:Y\to\N^\mathbb{Z}$ is an ESBC. Let us see that $\Phi$ is not of finite degree; suppose it is not. So there exists a barrier $\mathcal{B}$ on $Y$ attached to $\Phi$ such that for each $a\in\mathbb{N}$, the number of cylinders in $\mathcal{B}$ with $\Phi_0$-image equal to $a$ is finite. Let $C_Y(h_1),\cdots,C_Y(h_m)$ be the cylinders in $\mathcal{B}$ satisfying $\Phi_0(C_Y(h_i))=0$, $i=1,\cdots,m$. Consider, for each $j\in\mathbb{N}$, $x^j=(x^j_k)_{k\in\mathbb{Z}}$ with $x^j_k=0$ for $k\neq 1$ and $x^j_1=j$; clearly the $\Psi$-image $y^j=(y^j_k)_{k\in\mathbb{Z}}$ of $x^j$ is given by $y^j_0=2j$, $y^j_1=j$ and $y^j_k=0$ if $k\notin\{0,1\}$. Since $j$ is arbitrary, one has that $\{0,1\}\cap dom(h_i)=\emptyset$ and $h_i(\ell)=0$ for all $\ell\in dom(h_i)$ and every $i=1,\cdots,m$. On the other hand, if one takes any index $1\leq i\leq m$ and any $\ell\in dom(h_i)$, then $y=(y_k)_{k\in\mathbb{Z}}$, with $y_{\ell-1}=2$, $y_\ell=1$ and $y_k=0$ otherwise, belongs to $Y$ and $\Phi_0(y)=0$; however, $y\notin C_Y(h_i)$ whatever the index $i=1,\cdots,m$, which is a contradiction. 

\smallskip

Now we will show that $\Phi$ is not a sliding block code. We proceed on the contrary and assume without generality loss that $\Phi$ is induced by a local rule with same memory and anticipation, say $L\geq 1$; that is, there exists $\varphi:\mathcal{L}_{2L+1}(Y)\to\mathbb{N}$ such that 
$\Phi(y)_k=\varphi\left(y|_{[-L+k,k+L]}\right)$ for all $y\in Y$ and every $k\in\mathbb{Z}$; we note 
in particular that $\varphi\left(y|_{[-L,L]}\right)=x^y_0$ where $x^y$ is the unique element of $\mathbb{N}^\mathbb{Z}$ such that $\Phi(y)=x^y$. However, this assumption is negated by the following fact. For the elements $x=(x_k)_{k\in\mathbb{Z}}$ and $x'=(x_k')_{k\in\mathbb{Z}}$ in $\mathbb{N}^\mathbb{Z}$ given by
\begin{multline*}
x_k=\begin{cases}
2^{L+1-k},\text{ if $-L\leq k\leq L+1$}\\
0,\text{ otherwise}
\end{cases}
\,\text{ and }\, \\
x_k'=\begin{cases}
2^{2\ell+1},\text{ if $k=L-2\ell+1$ and $0\leq \ell \leq L$}\\
0,\text{ otherwise}
\end{cases}
\end{multline*}
one obtains, after a direct calculation, that $y=\Psi(x)$ and $z=\Psi(x')$ have the same central block
$y|_{[-L,L]}=z|_{[-L,L]}=2^{2L+2}2^{2L+1}\cdots 2^32^2$, but $x_0\neq x_0'$.

\smallskip

To finish, we construct a barrier on $Y$ attached to $\Phi$, on it we define an extended local rule inducing $\Phi$. For each $N\geq 1$ 
and $y=(y_n)_{n\in\mathbb{Z}}\in Y$ consider the set 
$$
S_N(y)=\{w_{-N}\cdots w_{N+1}\in\N^{2N+2}: w_i+2w_{i+1}=y_i,\,i=-N,\cdots, N\}.
$$
The next properties are immediate:

\smallskip
\noindent
(i) The block $x^y\vert_{[-N,N+1]}$ is always in $S_N(y)$.

\smallskip
\noindent
(ii) If 
$N\geq 2$, a block $w_{-N}\cdots w_{N+1}\in S_N(y)$ iff
$w_{-N+1}\cdots w_{N}$ belongs to $S_{N-1}(y)$, $w_{-N}+2w_{-N+1}=y_{-N}$ and $w_{N}+2w_{N+1}=y_{N}$.

\smallskip
\noindent
(iii) If  $w_{-N}\cdots w_{N+1}$ and $w_{-N}'\cdots w_{N+1}'$ are blocks in $S_N(y)$ with $w_i=w_i'$ for some $i$, then those blocks match.

\smallskip

Additionally, since for each $w_{-1}w_0w_1w_2\in S_1(y)$ one has $w_1=\frac{1}{2}(y_0-w_0)$, the cardinal $\# S_1(y)$ of $S_1(y)$ is bounded above by  
$\lfloor \frac{y_0}{2}\rfloor +1$, where $\lfloor x\rfloor$ means the integer 
part of $x$. From property (ii) above 
$\# S_2(y)\leq \lfloor \lfloor \frac{y_0}{2} \rfloor/2 \rfloor +1=\lfloor 
\frac{y_0}{4} \rfloor +1$. Thus, by a recursive argument one concludes that
$1\leq \# S_N(y)\leq \lfloor \frac{y_0}{2^N} \rfloor +1$, therefore there exists a 
first integer $r(y)\geq 1$ such that $\# S_N(y)=1$ for all $N\geq r(y)$. Obviously 
if $1\leq N<r(y)$, the system of diophantine equations 
$w_i+2w_{i+1}=y_i$, with $|i|\leq N$, has at least two solutions in 
$\N^{2N+2}$ and  $x_{-N}^y\cdots x_{N+1}^y$ is the unique solution of such system 
when $N\geq r(y)$. 

\smallskip

For each $y\in Y$ we consider the function 
$h^y:\{-r(y),\cdots,r(y)\}\to\N$ defined by $h^y(i)=y_i$ for all $0\leq 
i\leq r(y)$. It is straightforward to check that the collection $\mathcal{B}$ of the cylinders 
$C_Y(h^y)$ is a barrier attached on $Y$ to $\Phi$ and $\varphi:\mathcal{B}\to\mathbb{N}$ with $\varphi(C_Y(h^y))=x^y_0$ induces $\Phi$. Note that $r(z)=r(y)$ for all $z\in C_Y(h^y)$; indeed, 
either $C_Y(h^y)\cap C_Y(h^z)=\emptyset$ if $z\notin C_Y(h^y)$ or $C_Y(h^y)=C_Y(h^z)$ for all 
$z\in C_Y(h^y)$. 
\end{exa}
% 
% 
%\medskip
%\noindent 
%{\bf Acknowledgements.}
%The authors would like to thank the referee for his/her helpful comments and suggestions which led to the improvement of the manuscript.
%%%%%%%%%%%%%%%%%%%%%%%%%%%%%%%%%%%%%%%%%%%%%%%%%%%%%%%%%%%%%%%%%%

\end{document}